\documentclass[12pt]{amsart}
\usepackage{newlfont,amsmath,mathrsfs,enumerate,amssymb,array,amscd,xcolor}

\newcommand{\bbQ}{{\mathbb Q}}
\newcommand{\bbC}{{\mathbb C}}

\newcommand{\bbR}{{\mathbb R}}

\newcommand{\bbZ}{{\mathbb Z}}

\newcommand{\Ind}{{\mathrm{Ind}}}

\newcommand{\frakg}{{\mathfrak{g}}}

\newcommand{\frakp}{{\mathfrak{p}}}

\newcommand{\frakt}{{\mathfrak{t}}}

\newcommand{\Hom}{{\mathrm{Hom}}}

\newtheorem{lemma}{Lemma}[section]

\newtheorem{prop}[lemma]{Proposition}
\newtheorem{thm}[lemma]{Theorem}
\newtheorem{cor}[lemma]{Corollary}
\theoremstyle{definition}
\newtheorem{definition}{Definition}[section]
\newtheorem{example}{Example}[section]
\theoremstyle{remark}
\newtheorem{remark}{Remark}

\begin{document}


\title{principal series representations of metaplectic groups}
\author{Shiang Tang}
\address{Shiang Tang, Department of Mathematics, University of Utah,
  Salt Lake City, UT 84112} \email{tang@math.utah.edu}

\begin{abstract}
We study the principal series representations of central extensions of a split reductive algebraic group by a cyclic group of order $n$. We compute the Plancherel measure of the representation using Eisenstein series and a comparison method. In addition, we construct genuine central characters of the metaplectic torus in the simply-laced case.
\end{abstract}

\maketitle

\section{Introduction}
Let $F$ be a non-archimedean local field with ring of integers $\mathcal O_F$ and assume that $\underline G$ is a split reductive group defined over $F$. Let $n$ be a positive integer. A central extension of $\underline G$ by $\mu_n$ is a topological group $G$ that fits in the following exact sequence:
$$1\to \mu_n \to G \to \underline G \to 1$$
The aim of this paper is to study the principal series representations of $G$. Such representations have been studied by several authors. Kazhdan and Patterson \cite{KP84} have studied the case of $GL_n$, in which they have carried out the local theory and global theory, in particular, the theory of Eisenstein series in great detail. Savin \cite{Sav04} has considered the case when $G$ is simply laced and simply connected. The double cover of a general simply connected reductive group has been studied by Loke and Savin \cite{LS10} in which they construct genuine automorphic representations of the metaplectic torus using theta series.  
McNamara \cite{Mc12} has generalized the theory to the case of an arbitrary reductive group, under the hypothesis that $2n$ is coprime to the residual characteristic of $F$ and $F^{\times}$ contains $2n$ distinct $2n$-th roots of unity. In another direction, Weissman \cite{Weis16} has studied extensively the structure of the metaplectic torus, both in the local and global cases, and their representations. This paper studies the same objects the above authors have considered, with the following new results and constructions:
\begin{enumerate}
\item We remove the tameness assumption on $n$ imposed by McNamara and other authors. 
\item We use a comparison method to compute the Plancherel measures of ramified representations of $G(F)$. More precisely, we compare the covering group with a well-chosen linear group (motivated by \cite{Weis16}), and transfer the Plancherel measure from the linear case to the non-linear case. This has been done for $G$ the $n$-fold cover of $\mathrm{SL}_2(F)$ by Goldberg and Szpruch in \cite{GS16} using a different approach.
\item We give explicit constructions of genuine Weyl group invariant central automorphic character of the covering torus (in the sense of \cite{Weis16}, Definition 4.7). 
\end{enumerate}

The organization of this paper is as follows. In Section 2, we review the local theory, most of the results are known for a long time. In Section 3, we review the global theory, in particular, we obtain a product formula of local Plancherel measures which will be used in the comparison method mentioned above. In Section 4, we carry out the comparison method and prove our main theorem on the relation between Plancherel measures of non-linear groups and their associated linear groups, under the assumption on the existence of genuine Weyl group invariant central automorphic characters of the metaplectic torus. In Section 5, we construct explicitly such characters for simply-laced groups, and hence prove their existence.

I would like to thank my advisor Gordan Savin, without whom none of this work would be possible.

\section{central extensions of split reductive groups over local fields}

\subsection{Group theory}
Let $p$ be a fixed prime and $n$ be a fixed integer. Let $F$ be a finite extension of $\mathbb Q_p$ with ring of integers $R$ and maximal ideal $\mathfrak{m}$ generated by a uniformizer $\varpi$. We assume that $F$ contains the group $\mu_n$ of $n$-th roots of unity. Let $\underline G$ be a split reductive group over $F$ with Lie algebra $\mathfrak g$. Let $\underline T$ be a fixed maximal split torus of $\underline G$ with Lie algebra $\mathfrak t$. Let $\Phi=\Phi(\underline G,\underline T)$ be the root system associated to $\underline T$ and let $\Phi^{\vee}$ the set of coroots. We fix a Borel subgroup $\underline B$ containing $\underline T$ with the corresponding set $\Phi^+$ of positive roots and let $\Delta \subset \Phi^+$ be the subset of simple roots. Let $X$ be the character lattice and $Y$ be the cocharacter lattice of $\underline T$. We have $\Phi \subset X$ and $\Phi^{\vee} \subset Y$. The quadruple 
$$(X,\Phi,Y,\Phi^{\vee})$$ is called the root datum of $\underline G$, which determines $\underline G$ up to isomorphism.        
Let $H_{\beta}$ ($\beta\in\Delta$), $X_{\alpha}$ ($\alpha\in\Phi$) be a Chevalley basis of $\mathfrak g^{der}=[\mathfrak g,\mathfrak g]$. Let $\underline G^{sc}$ be the simply connected Chevalley group with Lie algebra $\mathfrak g^{der}$. There exists a map
$$\underline G^{sc} \to\underline  G$$ 
which differentiates to the inclusion map $\mathfrak g^{der} \to \mathfrak g$ on the level of Lie algebras. The elements $\underline x_{\alpha}(t)$, $\underline w_{\alpha}(t)$, $\underline h_{\alpha}(t)$ in $\underline G^{sc}$ are defined as in \cite{St67}. We use the same symbols to denote their images in $\underline G$. However, $\underline x_{\alpha}(t)$ for $t\in F, \alpha \in \Phi$ no longer generate $\underline G$. Let $\underline K$ be a maximal subgroup of $\underline G$ with the property that the adjoint action of $\underline K$ preserves the $R$-lattice in $\mathfrak g^{der}$ spanned by the Chevalley basis. Then $\underline K$ is a maximal compact subgroup of $\underline G$. In particular, when $\underline G$ is simply-connected, $\underline K$ is the subgroup of $\underline G$ generated by $\underline x_{\alpha}(t)$ for $t\in R$, $\alpha\in\Phi$.

We may identify $\underline T$ with $Y \otimes F^{\times}$. Let $(,):Y\times Y \to \mathbb Z$ be an integer valued symmetric bilinear form on $Y$ that is Weyl group invariant and satisfies $(\alpha^{\vee},\alpha^{\vee})=2$ when $\alpha$ is long.
Let $G$ be a nontrivial central extension of $\underline G$ by the group of $n$-th roots of unity:
$$1\rightarrow \mu_n \rightarrow G \to \underline G \rightarrow 1$$ Let $T$ be the preimage of $\underline T$ in $G$. We assume the followings:

\begin{enumerate}
\item $$[y_1\otimes t_1,y_2\otimes t_2]=(t_1,t_2)_n^{(y_1,y_2)}$$ where $y_i\otimes t_i$ represents an element in the fibre of $y_i\otimes t_i \in \underline T$ by abuse of notation and $(,)_n$ is the Hilbert norm residue symbol in class field theory.

\item If $p$ does not divide $n$, then there is a splitting (not necessarily unique) section $s:\underline K \to G$. The image of $s$ is denoted by $K$.
\end{enumerate}

\begin{remark}
Item 1 implies that given $t,t'\in T$, their commutator $[t,t'] \in \mu_n$ only depends on $\pi(t),\pi(t') \in \underline T$.
\end{remark}

\begin{remark}
Item 2 implies that on $\underline K_0$, the subgroup of $\underline K$ generated by $\underline x_{\alpha}(t)$ with $t\in R$ and $\alpha\in\Phi$, $s$ is the same as the canonical section which sends $\underline x_{\alpha}(t)$ to its unique lift in $G$.
\end{remark}

\begin{remark}
For any $\alpha \in \Phi$, $t \in F$,
$\underline x_{\alpha}(t)\in \underline G$ lifts uniquely to $G$, and we use $x_{\alpha}(t)$ to denote the lift. Define
\begin{align*}
&w_{\alpha}(t):=x_{\alpha}(t)x_{-\alpha}(-t^{-1})x_{\alpha}(t)\\
&h_{\alpha}(t):=w_{\alpha}(t)w_{\alpha}(-1)\\
\end{align*}
By \cite{St67}, when $\Phi$ is simply laced, we have the following relations in $T$: 
$$h_{\alpha}(t)h_{\alpha}(u)=(t,u)_nh_{\alpha}(tu)$$
$$[h_{\alpha}(t),h_{\beta}(u)]=(t,u)_n^{(\alpha^{\vee},\beta^{\vee})}$$

However, the $h_{\alpha}$'s do not generate $T$ unless $[G,G]=G$. 
\end{remark}

\subsection{Unramified principal series}
In this section, we keep the assumptions made in Section 2.1. Suppose we have an irreducible genuine representation $\sigma$ of the metaplectic torus $T$ which is Weyl group invariant, i.e., its central character is fixed by the conjugation action of the Weyl group. By \cite{AB07}, Proposition 2.2, there is a bijection between irreducible genuine representations of $T$ and genuine central characters of its center $Z_T$. We say that $\sigma$ is \emph{unramified} if $\sigma^{T_0}\neq 0$ where $T_0:=T\cap K$. 

\begin{prop}\label{2.1}
$Z_T\cdot T_0$ is a maximal abelian subgroup of $T$ containing $T_0$. 
\end{prop}
\begin{proof}
Clearly $Z_T\cdot T_0$ is an abelian group. We need to show that any element in $Z_T(T_0)$ must belong to $Z_T\cdot T_0$. Observe that every element in $T/T_0$ is represented by an element in $T$ which is of the form $x\otimes \varpi$ for some $x\in Y$. Therefore it is enough to show that if $x\otimes \varpi$ centralizes $T_0$, then it must lie in $Z_T$. To see this, note that 
$[x\otimes \varpi, y\otimes u]=1$ for any $y\in Y$ and any $u \in R^{\times}$, the LHS of which is $(\varpi,u)^{(x,y)}=(\bar u^{-1})^{(x,y)\cdot \frac{q-1}{n}}$ by Item 1. So $n$ divides $(x,y)$, for all $y \in Y$. It follows that $x\otimes \varpi \in Z_T$.  
\end{proof}

\begin{prop}\label{2.2}
If $p$ does not divide $n$ and $\sigma$ is unramified, then $\dim \sigma^{T_0}=1$.
\end{prop}
\begin{proof}
Let $\chi$ be the central character of $\sigma$.
Observe that $\chi$ is trivial on $Z_T\cap K$. In fact, choose $0\neq v_0\in \sigma^{T_0}$, we have for any $z\in T_0$, $v_0=\sigma(z)v_0=\chi(z)v_0$ which forces $\chi(z)$ to be $1$. Let $\tilde{\chi}$ be the extension of $\chi$ to $Z_T\cdot T_0$ by letting $\chi$ act trivially on $T_0$. Form the induction $$Ind_{Z_TT_0}^T(\tilde{\chi})$$ which has central character $\chi$. Frobenius reciprocity gives
$$Hom_T(Ind_{Z_TT_0}^T(\tilde{\chi}),\sigma)=Hom_{Z_TT_0}(\tilde{\chi},\sigma|_{Z_TT_0})$$ Since $\sigma$ is unramified, there is a line $L \subset \sigma$ that is fixed by $T_0$, so $\tilde{\chi}$ is a summand of $\sigma|_{Z_TT_0}$ and hence there is a nonzero $T$-intertwining map from $Ind_{Z_TT_0}^T(\tilde{\chi})$ to $\sigma$. 

Claim: For any $s\in T-Z_T\cdot T_0$, $\tilde{\chi}^s \neq \tilde{\chi}$ on $T_0$.

Once we have proved this claim, it will follow by Mackey's criterion that $Ind_{Z_TT_0}^T(\tilde{\chi})$ is irreducible and isomorphic to $\sigma$, and that $\tilde{\chi}$ has multiplicity one in $Ind_{Z_TT_0}^T(\tilde{\chi})|_{Z_TT_0}$, from which the proposition follows.

To prove the claim, we may assume that $s=x\otimes \varpi$ for some $x \in Y$. If the claim were false, then $\tilde{\chi}([s,t])=1$ for all $t \in T_0$, that is, $[x\otimes \varpi, y\otimes u]=1$ for any $y\in Y$ and $u \in R^{\times}$. By the proof of the previous proposition, this would imply $s \in Z_T$, a contradiction. 
\end{proof}

\begin{lemma}\label{2.3}
Suppose $\alpha, \beta$ are roots in $\Phi$ such that $<\beta,\alpha>=-1$ (which implies that $|\alpha|\geq |\beta|$ and $w_{\alpha}\beta=\alpha+\beta$), then 
$$h_{\alpha+\beta}(t)=(c(\alpha,\beta)t^{<\alpha,\beta>},t^{n_{\beta}})h_{\alpha}(t^{-<\alpha,\beta>})h_{\beta}(t)$$ where $n_{\beta}$ is 1 when $\beta$ is long and is the maximal number of links between adjacent dots in the Dynkin diagram of $\Phi$ when $\beta$ is short.
\end{lemma}
\begin{proof}
Let $u=t^{-<\alpha,\beta>}$.
\begin{align*}
&\:\:\:\:\: h_{\alpha+\beta}(t)\\
&=h_{w_{\alpha}\beta}(t)\\
&=\zeta w_{\alpha}(u)h_{\beta}(t)w_{\alpha}(-u)\\
&=\zeta w_{\alpha}(u)h_{\beta}(t)x_{\alpha}(-u)x_{-\alpha}(u^{-1})x_{\alpha}(-u)\\
&=\zeta w_{\alpha}(u)x_{\alpha}(-t^{<\alpha,\beta>}u)h_{\beta}(t)x_{-\alpha}(u^{-1})x_{\alpha}(-u)\\
&=\zeta w_{\alpha}(u)x_{\alpha}(-t^{<\alpha,\beta>}u)x_{-\alpha}(t^{-<\alpha,\beta>}u^{-1})h_{\beta}(t)x_{\alpha}(-u)\\
&=\zeta w_{\alpha}(u)x_{\alpha}(-t^{<\alpha,\beta>}u)x_{-\alpha}(t^{-<\alpha,\beta>}u^{-1})x_{\alpha}(-t^{<\alpha,\beta>}u)h_{\beta}(t)\\
&=\zeta w_{\alpha}(u)w_{\alpha}(-t^{<\alpha,\beta>}u)h_{\beta}(t)\\
&=\zeta h_{\alpha}(t^{-<\alpha,\beta>})h_{\beta}(t)\\
\end{align*}
where $\zeta\in \mu_n$.

We now compute $\zeta$:
By \cite{St67}, 
$$w_{\alpha}(u)h_{\beta}(t)w_{\alpha}(-u)=h_{w_{\alpha}\beta}(c(\alpha,\beta)ut)h_{w_{\alpha}\beta}(c(\alpha,\beta)u)^{-1}$$
$$=(c(\alpha,\beta)u^{-1},c(\alpha,\beta)u)^{n_{\beta}}h_{w_{\alpha}\beta}(c(\alpha,\beta)ut)h_{w_{\alpha}\beta}(c(\alpha,\beta)u^{-1})$$
$$=(cu^{-1},cu)^{n_{\beta}}(cut,cu^{-1})^{n_{\beta}}h_{w_{\alpha}\beta}(t)=(cu^{-1},t^{-1})^{n_{\beta}}h_{w_{\alpha}\beta}(t)$$
So $\zeta=(cu^{-1},t)^{n_{\beta}}=(ct^{<\alpha,\beta>},t)^{n_{\beta}}$. 
\end{proof}

\begin{prop}\label{2.4} Suppose that $\Phi$ is not of type $A_1$ and $\alpha\in \Phi$ is long, then 
$\sigma(h_{\alpha}(t))=1$ for any $\alpha\in\Phi$, $t\in (F^{\times})^n$.
\end{prop} 

\begin{proof}
We have
$$w_{\alpha}(1)h_{\alpha}(t)w_{\alpha}(1)=h_{\alpha}(t^{-1})$$
$$w_{\gamma}(1)h_{\alpha}(t)w_{\gamma}(-1)=h_{w_{\gamma}\alpha}(t)$$ for any $t\in (F^{\times})^n$ and any $\alpha,\gamma\in\Phi$.
Since the central character of $\sigma$ is invariant under the action of $W$, the above identities imply $\sigma(h_{\alpha}(t))=\pm 1$ for any $t\in (F^{\times})^n$ and for all roots $\alpha\in\Phi$ of the same length.

Suppose that $\Phi$ is simply-laced and is not of type $A_1$. We have by \ref{2.3}, $h_{\alpha+\beta}(t)=h_{\alpha}(t)h_{\beta}(t)$ when $<\beta,\alpha>=-1$ and $t\in (F^{\times})^n$. It is clear that such pair of roots exist if $\Phi$ is not of type $A_1$. Pick such pair of roots $\alpha, \beta$, we have $\sigma(h_{\alpha+\beta}(t))=\sigma(h_{\alpha}(t))=\sigma(h_{\beta}(t))=1$. Therefore $\sigma(h_{\alpha}(t))=1$, $t\in (F^{\times})^n$ for all roots $\alpha$.       

The demonstration for $\Phi$ multi-laced is similar. 
\end{proof}

\begin{remark}
\ref{2.4} is true for short roots for $\Phi$ of type $G_2$ and $3$ not dividing $n$, and for $\Phi$ of type $F_4$ and $2$ not dividing $n$.
\end{remark}

\begin{example}
If $\underline G=SL_2$ and $n=1$, any quadratic character of $T$ is Weyl group invariant. So \ref{2.4} fails in this case.
\end{example}

Let $\underline B$ be a Borel subgroup of $\underline G$ corresponding to $\Phi^+$, we have $\underline B=\underline TN$ where $N$ is the unipotent radical of $\underline{B}$ which lifts uniquely to $G$. Let $B=TN$ be the preimage of $B=\underline TN$. Define an induced representation (normalized induction)
$$I(\nu)=Ind_B^G(\nu\otimes\sigma)$$ where $\nu$ is an unramified character of $\underline T$. 
For any $w\in N_K(T)/T_0$, choose a representative $\hat w$ in $N_K(T)$. There is a unique isomorphism $$j_{\hat w}:\sigma^{\hat w} \to \sigma $$ such that $j_{\hat w}$ fixes $\sigma^{T_0}$ pointwise. In other words, $j\circ \sigma(\hat w^{-1}t\hat w)=\sigma(t)\circ j$, $\forall t\in T$ and $j(v_0)=v_0$ for any $v_0 \in \sigma^{T_0}$. 
If $\hat w'$ is a different representative in $N_K(T)$, say $\hat w'=\hat wt_0$ for some $t_0 \in T_0$, then $j_{\hat w'}=j_{\hat w}\circ \sigma(t_0)$.  

Fix an Haar measure on $N$. 
For every $w\in W$, there is an intertwining operator 
$$A_w=A_w(\nu):I(\nu) \to I(\nu^w)$$ defined as follows:

$$A_w(f)(x)=j_{\hat w}\int_{N_w} f(\hat w^{-1}nx)dn $$ where $\hat w \in N_K(T)$ is a representative of $w$.

\begin{lemma}
$A_w$ is well-defined and is independent of the choice of a representative of $w$.
\end{lemma}
\begin{proof}
\begin{align*}
&\:\:\:\:\:  (A_wf)(tx) \\
&=j_{\hat w}\int_{N_w}f(\hat w^{-1}ntx)dn\\
&=j_{\hat w}\int_{N_w}f((\hat w^{-1}t\hat w)\hat w^{-1}(t^{-1}nt)x)dn\\
&=\delta^{1/2}(\hat w^{-1}t\hat w)\chi(\hat w^{-1}t\hat w)j_{\hat w}\sigma(\hat w^{-1}t\hat w)\int_{N_w}f(\hat w^{-1}(t^{-1}nt)x)dn\\
&=\delta^{1/2}(\hat w^{-1}t\hat w)\chi(\hat w^{-1}t\hat w)\sigma(t)j_{\hat w}\int_{N_w}f(\hat w^{-1}(t^{-1}nt)x)dn\\
&=\delta^{1/2}(t)\chi^w(t)\sigma(t)j_{\hat w}\int_{N_w}f(\hat w^{-1}nx)dn\\
&=\delta^{1/2}(t)\chi^w(t)\sigma(t)(A_wf)(x)\\
\end{align*}

So $A_w$ is well-defined. Since $j_{\hat w'}=j_{\hat w}\circ \sigma(t_0)$, $A_w$ is independent of the choice of $\hat w$. 

\end{proof}

If $\sigma$ is unramified, then $I(\nu)^K$ is nonzero. Moreover, it has dimension one by \ref{2.2}. The induced map
$$I(\nu)^K \to I(\nu^w)^K$$ can be computed explicitly. See \cite{LS10}, Section 6 for the case when $G$ is the double cover of a simply connected algebraic group over non-archimedean fields, and see \cite{Mc12}, Section 7 for the case when $G$ is an $n$-fold cover of an algebraic group over non-archimedean fields. In \cite{Tang17}, an archimedean analogue is also established. The following proposition is well-known, we record it here for the reader's convenience. We normalize the Haar measure $m$ on $F$ such that $m(R)=d_F^{-1/2}$ where $d_F$ is the discriminant of $F$ over $\mathbb Q_p$.

\begin{prop}\label{2.6}
Suppose that $\alpha$ is a long root when $\Phi$ is not of type $A_1$ and suppose $\sigma$ satisfies \ref{2.4} when $\Phi$ is of type $A_1$.
Let $f_{\nu}^o \in I(\nu)^K$ be a function such that $f_{\nu}^o(1)=v_0$ where $v_0$ is a fixed nonzero vector in $\sigma^{T_0}$. Let $w=w_{\alpha}$. Then $A_wf_{\nu}^o=c(\nu)f_{\nu^{w}}^o$ where
$$c(\nu)=d_F^{-1/2}\frac{1-q^{-1}\nu(h_{\alpha}(\varpi^n))}{1-\nu(h_{\alpha}(\varpi^n))}$$
\end{prop}

\begin{proof}
This is essentially a calculation in the group $SL_2$. To alleviate notation, we use $f$ to denote $f_{\nu}^o$. 

$$\int_F f(w_{\alpha}(-1)x_{\alpha}(t))dt=m(R)v_0+\sum_{k\geq 1}\int_{\varpi^{-k}R^{\times}} f(w_{\alpha}(-1)x_{\alpha}(t)) dt$$ 

For $t \in F-R$, 
\begin{align*}
&\:\:\:\:\:f(w_{\alpha}(-1)x_{\alpha}(t))\\
&=f(x_{\alpha}(t^{-1})w_{\alpha}(-1)x_{\alpha}(t))\\
&=f(w_{\alpha}(-1)x_{-\alpha}(-t^{-1})x_{\alpha}(t))\\
&=f(w_{\alpha}(-1)x_{-\alpha}(-t^{-1})x_{\alpha}(t)x_{-\alpha}(-t^{-1}))\\
&=f(w_{\alpha}(-1)w_{-\alpha}(-t^{-1}))=f(w_{\alpha}(-1)w_{\alpha}(t))\\
&=f(w_{\alpha}(-1)h_{\alpha}(t)w_{\alpha}(1))\\
&=f(h_{\alpha}(t^{-1}))\\
\end{align*}

Let $t^{-1}=\varpi^ku$, $u \in R^{\times}$, then 
$$h_{\alpha}(t^{-1})=(\varpi^k,u)_n^{-1}h_{\alpha}(\varpi^k)h_{\alpha}(u)=(\bar u)^{k\frac{q-1}{n}}h_{\alpha}(\varpi^k)h_{\alpha}(u)$$
where $(\bar u)^{k\frac{q-1}{n}} \in \mathbb F_q^{\times}$ is a $n$-th root of unity in $\mathbb F_q$. The expression is well-defined because $n|q-1$. 
We have
\begin{align*}
&\:\:\:\:\:f(h_{\alpha}(t^{-1}))\\
&=(\bar u)^{k\frac{q-1}{n}}\delta^{1/2}\nu(h_{\alpha}(\varpi^k))\sigma(h_{\alpha}(\varpi^k))v_0\\
&=(\bar u)^{k\frac{q-1}{n}}q^{-k}\nu(h_{\alpha}(\varpi^k))\sigma(h_{\alpha}(\varpi^k))v_0\\
\end{align*}

It is easy to see that, for any $m \in\mathbb Z$, $\sum_{x \in \mathbb F_q}x^m=0$ if $q-1$ does not divide $m$. 

Therefore, if $n$ does not divide $k$, $$\int_{\varpi^{-k}R^{\times}} f(h_{\alpha}(t^{-1})) dt=0$$ and if $n|k$, then by \ref{2.4}, 
$$f(h_{\alpha}(t^{-1}))=q^{-k}\nu(h_{\alpha}(\varpi^k))\sigma(h_{\alpha}(\varpi^k))v_0=q^{-k}\nu(h_{\alpha}(\varpi^k))v_0$$ So 
$$\int_{\varpi^{-k}R^{\times}} f(h_{\alpha}(t^{-1})) dt=m(\varpi^{-k}R^{\times})q^{-k}\nu(h_{\alpha}(\varpi^k))v_0$$
$$=d_F^{-1/2}(1-\frac{1}{q})\nu(h_{\alpha}(\varpi^k))v_0$$

Finally,
\begin{align*}
&\:\:\:\:\:\int_F f(w_{\alpha}(-1)x_{\alpha}(t))dt\\
&=d_F^{-1/2}v_0+\sum_{n|k} d_F^{-1/2}\frac{q-1}{q}\nu(h_{\alpha}(\varpi^k))v_0\\
&=d_F^{-1/2}v_0+d_F^{-1/2}\frac{q-1}{q}\frac{\nu(h_{\alpha}(\varpi^n))}{1-\nu(h_{\alpha}(\varpi^n))}v_0\\
&=d_F^{-1/2}\frac{1-q^{-1}\nu(h_{\alpha}(\varpi^n))}{1-\nu(h_{\alpha}(\varpi^n))}v_0\\
\end{align*}

\end{proof}

\section{Central extensions of split reductive groups over the adeles}

\subsection{Group theory}
Given a number field $k$ containing $\mu_n$, let $\Sigma$ be the set of all places of $k$. Let $\underline G$ be a split simply-connected Chevalley group defined over $k$. We use $\underline G_v$ to denote $\underline G(k_v)$. For each $v$, there is a unique (up to isomorphism) nontrivial central extension $G_v$ of $\underline G_v$ by $\mu_n$ (\cite{St67}): 
$$1\to \mu_n \to G_v \to \underline G_v \to 1$$

Consider the restricted product
$$\mathcal G=\prod_{v\in \Sigma} G_v$$ with respect to the family of compact subgroups $K_v$ of $G_v$ generated by $x_{\alpha}(t)$ ($\alpha\in \Phi, t \in r_v$). For each $v$, let $Z_v$ denote the kernel of the covering projection $G_v\to \underline G_v$. Form the restricted product (i.e., all but finitely many components of any given element is 1)
$\mathcal Z=\prod_{v\in \Sigma}Z_v$ 
and let $\mathcal Z_0$ be the subgroup of $\mathcal Z$ consisting of elements whose coordinates multiply up to one. Define 
$$G(\mathbb A_k):=\mathcal G/\mathcal Z_0$$
which fits into a short exact sequence:
$$1 \to \mu_n \to G(\mathbb A_k) \to \underline G(\mathbb A_k) \to 1$$

\begin{prop}
Suppose that $\underline G$ is a simply-connected Chevalley group. Then there is an injective homomorphism from $\underline G(k)$ to $G(\mathbb A_k)$.
\end{prop}
\begin{proof}
Let $\mathcal G(k)$ be the subgroup of $\mathcal G$ generated by $X_{\alpha}(t)=({x}_{\alpha}(t),{x}_{\alpha}(t),...)$ with $\alpha\in\Phi$, $t\in k$ and let $H(k)$ be its image in $G(\mathbb A_k)$. We still use $X_{\alpha}(t)$ to denote the image of $X_{\alpha}(t)$ in $H(k)$. They generate $H(k)$ by definition. Put $W_{\alpha}(t)=X_{\alpha}(t)X_{-\alpha}(-t^{-1})X_{\alpha}(t)$ and $H_{\alpha}(t)=W_{\alpha}(t)W_{\alpha}(-1)$.   

Claim: $$\underline G(k) \cong H(k)$$

To prove the claim, notice that $H_{\alpha}(t)H_{\alpha}(u)=H_{\alpha}(tu)$ by Hilbert reciprocity. Since $\underline G$ is simply connected, there is a surjection $f:\underline G(k) \to H(k)$ mapping $x_{\alpha}(t)$ to $X_{\alpha}(t)$. On the other hand, there is a surjection $g:\mathcal G(k)\to \underline G(k)$ such that $g(X_{\alpha}(t))=x_{\alpha}(t)$ and that $g$ is trivial on $\mathcal Z_0$. So $g$ factors through $H(k)\to \underline G(k)$ which is still denoted by $g$. Now $g\circ f$ is the identity map, which forces $f,g$ to be isomorphisms.
\end{proof}

In general, if $\underline G$ is a split connected reductive group defined over $k$, such an explicit construction is not available to us. Instead, we will simply impose a number of axioms.

\begin{enumerate}
\item There is a central extension $G(\mathbb A_k)$ of $\underline G(\mathbb A_k)$ by $\mu_n$:
$$1 \to \mu_n \to G(\mathbb A_k) \to \underline G(\mathbb A_k) \to 1$$
\item For each place $v$ of $k$, there is a central extension $$1\to \mu_n \to G_v \to \underline G_v \to 1$$ such that [2.1, Item 1] holds for all $v$ and [2.1, Item 2] holds for all but finitely many $v$.
\item There is a surjection $$\prod_{v\in \Sigma} G_v \to G(\mathbb A_k)$$ whose kernel is in the center of $\prod_{v\in \Sigma} G_v$.
\item There is an injective homomorphism from $\underline G(k)$ to $G(\mathbb A_k)$.
\end{enumerate}

The metaplectic torus $T(\mathbb A_k)$ is the image of the restricted product of $\{T(k_v)\}_{v \in \Sigma}$ with respect to $\{T(r_v)=T(k_v) \cap K_v\}_{v \in \Sigma}$ in $G(\mathbb A_k)$.

\subsection{Principal series representations and intertwining maps}
We begin with a definition:
\begin{definition}
We say a character $\chi=\prod \chi_v$ of $Z_T=Z_T(\mathbb A_k)$ is a genuine automorphic central character (genuine ACC in short) of $T=T(\mathbb A_k)$ if 
\item 1. $\chi(\zeta z)=\zeta \chi(z)$, $\forall z \in Z_T$, $\forall \zeta \in \mu_n$.
\item 2. $\chi$ is trivial on $\underline T(k)\cap Z_T$.
\item 3. $\chi$ is unramified almost everywhere, i.e., for $v$ outside a finite set of places $S$ which contains all archimedean places and places where the residue characteristic divides $n$, $\chi_v$ is trivial on $Z_T(r_v)$.
\end{definition}

Suppose that $\chi=\prod \chi_v$ is invariant under the conjugation action of the Weyl group and $\sigma$ is an irreducible representation of $T(\mathbb A_k)$ with central character $\chi$.
$\sigma$ may be decomposed as a restricted tensor product of genuine irreducible representations $\sigma_v$ of $T_v$ with central character $\chi_v$:
$$\sigma=\bigotimes_{v\in \Sigma} \sigma_v$$
where the the restricted product is with respect to a fixed set of vectors $\{ u_v \neq 0 | v\notin S, u_v\in \sigma_v^{T_v^0} \}$. 

Let $\nu=\prod_v\nu_v$ be an unramified character of $\underline T(k) \backslash \underline T(\mathbb A_k)$. For example, the adelic absolute value function $|\cdot|_{\mathbb A_k}=\prod_v|\cdot|_v: \mathbb A_k^{\times} \to \mathbb R_{>0}$ is an unramified character of $\mathbb A_k^{\times}$. Let 
$$I(\nu)=Ind_{B(\mathbb A)}^{G(\mathbb A)}(\nu\otimes \sigma)$$ 
and let
$$I(\nu_v)=Ind_{B(k_v)}^{G(k_v)}(\nu_v \otimes \sigma_v)$$
We have $\dim I_v(\nu)^{K_v}=1$ for almost all $v$. For such $v$, let $f_{\nu}^v \in I_v(\nu)^{K_v}$ be the unique (vector valued) function so that $f_{\nu}^v(1)=u_v$. Then 
$I(\nu)$ is the restricted tensor product of $I(\nu_v)$ for $v\in \Sigma$ with respect to this set of functions. 

Suppose that $\sigma$ is a submodule of $L^2(\underline{T}(k)\backslash T(\mathbb A_k))$.
For a $w\in W$ and a choice $\hat w$ of its representative in $N_{\underline{G}(k)}\underline T(k)$, 
The natural action of $\hat w$ on $L^2(\underline{T}(k)\backslash T(\mathbb A_k))$ maps $\sigma$ to $\sigma^{\hat w}$ (which has the same underlying vector space but a twisted $T(\mathbb A_k)$-action). Let $j_{\hat w}$ be the vector space isomorphism
\begin{align*}
&\sigma^{\hat w}\to \sigma\\
&\varphi \mapsto \varphi^{\hat w}(t)=\varphi(\hat w^{-1}t\hat w)\\
\end{align*}
We have $j_{\hat w}=\otimes_v j_{\hat w}^v$ where $j_{\hat w}^v: \sigma_v^{\hat w} \to \sigma_v$ is defined in Section 2.2 for $v\notin S$.

Take a representative $\hat w$ of $w$, define for any $f\in I(\nu)$
$$\big(A_{w}(\nu)(f)\big)(x)=j_{\hat w}\int_{N_w(\mathbb A_k)} f(\hat w^{-1}nx)dn $$
Here $N_w(\mathbb A_k)=(N(\mathbb A_k) \cap wN(\mathbb A_k)w^{-1})\backslash N(\mathbb A_k)$ and we normalize the Haar measure on $\mathbb A_k$ so that it is the product measure of the Lebesgue measure on $\mathbb R$ for $v$ real, twice the Lebesgue measure on $\mathbb C$ for $v$ complex, and the unique Haar measure on $k_v$ such that $r_v$ has measure $d_{k_v}^{-1/2}$ for $v$ non-archimedean, where $d_{k_v}$ is the absolute discriminant of $k_v$. 

\begin{prop}
$A_{w}(\nu)$ maps $I(\nu)$ into $I(\nu^w)$ and is independent of the choice of a representative of $w$ in $N_{\underline{G}(k)}\underline T(k)$.
\end{prop}
\begin{proof}
By the definition of $j$, $A_{w}(\nu)$ maps $I(\nu)$ into $I(\nu^w)$. If $\hat w'=\hat wt$ for some $t \in \underline T(k)$, then $f(\hat w'^{-1}nx)=\delta_N^{1/2}(t)^{-1}\nu(t)^{-1}\sigma(t)^{-1}f(\hat w^{-1}nx)=\sigma(t)^{-1}f(\hat w^{-1}n x)$. On the other hand, $j_{\hat w'}=j_{\hat w}\circ \sigma(t)$ since any function in the vector space $\sigma$ is left $\underline T(k)$-invariant. Thus $A_w(\nu)$ is independent of the choice of a representative of $w$ in $N_{\underline{G}(k)}\underline T(k)$.  
\end{proof}

\begin{prop}
$$A_{w}(\nu)=\bigotimes_v A_{w}(\nu_v)$$ where for each $v$, $$A_{w}(\nu_v): I(\nu_v) \to I(\nu_v^w)$$ is an intertwining map between local principal series.  
\end{prop}
\begin{proof}
This follows from the definition of intertwining operators and the  identity $j_{\hat w}=\otimes j_{\hat w}^v$.
\end{proof}

\begin{remark}
The tensor decomposition of $A_w(\nu)$ also yields for $v \in S$ a canonical intertwining operator at $v$ (i.e., it does not depend on the choice of a representative of $w$).
\end{remark}

Let 
$$C^{\infty}(N(\mathbb A_k)\underline T(k) \backslash G(\mathbb A_k))$$ be the space of smooth, left $N(\mathbb A_k)$ invariant, left $\underline T(k)$-invariant, right $K(\mathbb A_k)$-finite functions on $G(\mathbb A_k)$, and let  
$$C^{\infty}(N(\mathbb A_k)\underline T(k) \backslash G(\mathbb A_k))_{\nu\otimes \sigma}$$ be the subspace consisting of functions $\phi$ such that for any $k\in K(\mathbb A_k)$, the function $\phi_k(t)=\delta^{-1/2}(t)\phi(tk)$ belongs to the isotypic component of $\nu\otimes \sigma$ in $L^2(\underline{T}(k)\backslash T(\mathbb A_k))$. (This definition is from \cite{MW95}, Ch 1, 3.3)

\begin{lemma}
For $f \in I(\nu)$, let $f^*$ to be the associated $\mathbb C$-valued function given by
$f^*(x)=(f(x))(1)$. Then $f^*\in C^{\infty}(N(\mathbb A)\underline T(k) \backslash G(\mathbb A))_{\nu\otimes \sigma}$.
\end{lemma}
\begin{proof}
First, $f^*\in C^{\infty}(N(\mathbb A)\underline T(k) \backslash G(\mathbb A))$: $N(\mathbb A)$-invariance of $f^*$ is obvious. For $\underline T(k)$-invariance, we note that for $t\in \underline T(k)$, $f^*(tx)=f(tx)(1)=(\sigma(t)f(x))(1)=f(x)(t)$ which is $f(x)(1)$ since $f(x)\in \sigma$. 
It remains to check that $\nu^{-1}f^*_k\in\sigma$ for any $k\in K(\mathbb A)$. In fact, $\forall t \in T(\mathbb A)$, $\nu^{-1}(t)\delta(t)^{-1/2}f^*(tk)=\nu^{-1}(t)\delta(t)^{-1/2}f(tk)(1)=(\sigma(t)f(k))(1)$
$=f(k)(t)$ which shows $\nu^{-1}f^*_k=f(k)\in \sigma$.
\end{proof}

For $\phi\in C^{\infty}(N(\mathbb A)\underline T(k) \backslash G(\mathbb A))_{\nu\otimes \sigma}$ and $x\in G(\mathbb A)$, whenever the integral below is convergent, we set 
$$\big(M_{w}(\nu)\phi\big)(x)=\int_{N_w(\mathbb A_k)} \phi(w^{-1}nx)dn$$ note that this integral is independent of the choice of a representative of $w$.

\begin{prop}
The integral defined above converges absolutely when $\nu$ belongs to an open subset of  $X^*(T)\otimes \mathbb C$ and $M_w(\nu)$ maps $C^{\infty}(N(\mathbb A)\underline T(k) \backslash G(\mathbb A))_{\nu\otimes \sigma}$ into $C^{\infty}(N(\mathbb A)\underline T(k) \backslash G(\mathbb A))_{\nu^w\otimes \sigma}$.
\end{prop}
\begin{proof}
The first half is proved in \cite{MW95}, Ch.2, 1.6 which also proved that $M_w(\nu)$ maps the space
$C^{\infty}(N(\mathbb A)\underline T(k) \backslash G(\mathbb A))_{\nu\otimes \sigma}$ into $C^{\infty}(N(\mathbb A)\underline T(k) \backslash G(\mathbb A))_{\nu^w\otimes \sigma^{w}}$. But the latter is equal to 
$$C^{\infty}(N(\mathbb A)\underline T(k) \backslash G(\mathbb A))_{\nu^w\otimes \sigma}$$ since $\sigma\cong \sigma^w$. 
\end{proof}

\begin{thm}(\cite{MW95}, Ch.4, 1.10)
$\nu \mapsto M_w(\nu)$ can be continued in a unique way to a meromorphic function on $X^*(T)\otimes \mathbb C$. Moreover, we have the equality of meromorphic functions given by 
$$M_{w'}(\nu^w)\circ M_w(\nu)=M_{w'w}(\nu)$$
\end{thm}

\begin{lemma}
For any $f\in I(\nu)$, we have 
$$(A_w(\nu)f)^*=M_w(\nu)f^*$$
\end{lemma}
\begin{proof}
\begin{align*}
&\:\:\:\:\:(A_w(\nu)f)^*(x)\\
&=\big(j_{\hat w}\int_{N_w(\mathbb A_k)} f(\hat w^{-1}nx)dn\big)(1)\\
&=\int_{N_w(\mathbb A_k)} (j_{\hat w}f(\hat w^{-1}nx))(1)dn\\
&=\int_{N_w(\mathbb A_k)} f(\hat w^{-1}nx)(1)dn\\
&=(M_w(\nu)f^*)(x)\\
\end{align*}

\end{proof}

\begin{cor}\label{3.8}
$$A_{w'}(\nu^w)\circ A_w(\nu)=A_{w'w}(\nu)$$
\end{cor}

\subsection{Rank one group}
In this section, we assume that $\frakg^{der}\cong \mathfrak{sl}_2$ so $\mathfrak g=\mathfrak t \oplus \mathfrak g_{\alpha} \oplus \mathfrak g_{-\alpha}$ where $\frakt$ is the maximal abelian subalgebra of $\frakg$ and $\alpha$ is a root of $\frakg$. Let $\underline G$ be a connected reductive group with Lie algebra $\frakg$. Let $\underline T$ be the maximal torus of $\underline G$ with Lie algebra $\frakt$. Let $w$ be the nontrivial element in the Weyl group $W \cong \mathbb Z/2\mathbb Z$. There is a homomorphism from $SL_2$ to $\underline G$ which differentiates to the inclusion map among the Lie algebras $\frakg^{der} \to \mathfrak g$. We use $\underline h(t)$ to denote the image of $diag(t,t^{-1}) \in SL_2$. Let $\nu: \underline T(\mathbb A_k) \to \mathbb R_{>0}$ be a character such that $\nu(\underline h(t))=|t|_{\mathbb A}^s$ for some complex number $s$.
Let $v$ be a place of $k$, for generic $\nu$, $I(\nu_v)$ is irreducible, hence by Schur's lemma, $A(\nu_v^w) \circ A(\nu_v)$ is a constant, which only depends on $s$ and we denote by $\lambda_v(s)$. Define the Plancherel measure $$\mu_v(s)=\frac{1}{\lambda_v(s)}$$ By \ref{2.6}, 
\begin{align*}
\mu_v(s)&=d_{k_v}\frac{1-\nu(\underline h(\varpi_v^n))}{1-q_v^{-1}\nu(\underline h(\varpi_v^n))}\frac{1-\nu(\underline h(\varpi_v^{-n}))}{1-q_v^{-1}\nu(\underline h(\varpi_v^{-n}))}\\
&=d_{k_v}\frac{1-q_v^{-ns}}{1-q_v^{-ns-1}}\frac{1-q_v^{ns}}{1-q_v^{ns-1}}\\
\end{align*}

We record here a basic property of $\mu_v$:
\begin{prop}\label{3.9}
For any non-archimedean place $v$ of $k$, $\mu_{v}(s)$ is a rational function of $q_v^s$, $\mu_{v}(s)=\mu_{v}(-s)$ and $\mu_v(it)\geq 0$ for $t$ real. 
\end{prop}

\begin{prop}\label{3.10} Assume that for all $v\notin S$, $\sigma(h(t^n))=1$, $\forall t\in k_v^{\times}$. Then
$$\left(\prod_{v\notin S} d_{k_v} \right)\left(\prod_{v\in S}\mu_v(s)\right)\frac{L_S(ns+1)}{L_S(ns)}\frac{L_S(-ns+1)}{L_S(-ns)}=1$$ where $$L_S(s)=\prod_{v\notin S}\frac{1}{1-q_v^{-s}}$$ is a partial $L$-function which converges when $Re(s)>>0$ and has a meromorphic continuation to $\bbC$.
\end{prop}
\begin{proof}
Take $f_{\nu} \in I(\nu)=\bigotimes_v I(\nu_v)=\otimes_v f_{\nu}^v$ such that $f_{\nu}^v \in I(\nu_v)$ for $v \in S$ and that $f_{\nu}^v \in I(\nu_v)^{K_v}$ for $v \notin S$.
\begin{align*}
A(\nu)f_{\nu}&=\prod_{v\in S}A(\nu_v)f_{\nu}^v\prod_{v \notin S}A(\nu_v)f_{\nu}^v\\
&=\prod_{v\in S}A(\nu_v)f_{\nu}^v\cdot \prod_{v\notin S} d_{k_v}^{-1/2} \frac{L_S(ns)}{L_S(ns+1)}\prod_{v \notin S}f_{\nu^w}^v\\
\end{align*}

Letting $A(\nu^w)$ act on both sides, we obtain by \ref{3.8} that 
\begin{align*}
f_{\nu}&=\prod_{v\in S}\lambda_v(s)f_{\nu}^v\cdot \left(\prod_{v\notin S} d_{k_v}^{-1}\frac{L_S(ns)}{L_S(ns+1)}\frac{L_S(-ns)}{L_S(-ns+1)}\right)\prod_{v \notin S}f_{\nu}^v \\
& =\left(\prod_{v\in S}\lambda_v(s)\prod_{v\notin S} d_{k_v}^{-1}\frac{L_S(ns)}{L_S(ns+1)}\frac{L_S(-ns)}{L_S(-ns+1)}\right) f_{\nu}\\
\end{align*}
which proves the proposition.
\end{proof}

\section{A comparison method}

\subsection{A related linear reductive group}
Given a connected reductive group $\underline G$ defined over a local field $F$ and a central extension $G$ of $\underline G$ by $\mu_n$, we will define a linear reductive group that is closely related to $G$.
As in Section 2.1, let $X=X^*(\underline T)$, $Y=X_*(\underline T)$ and fix a non-degenerate symmetric Weyl group invariant bilinear form $(,)$ on $Y$ such that $(\alpha^{\vee},\alpha^{\vee})=2$ for all $\alpha^{\vee} \in \Phi^{\vee}$. We define, following \cite{Weis16}, Section 1, 
$$Y^{\sharp}=\{y\in Y|(y,y') \in n\mathbb Z, \forall y'\in Y \}$$
It is clear that $Y^{\sharp} \otimes F^{\times}\cong \pi(Z_T)$ with $\pi$ the covering map. 
Let $Y'=\frac{1}{n}Y^{\sharp}$ and let $X'$ be the dual lattice of $Y'$ under the natural perfect pairing $<,>: X_{\mathbb Q}\times Y_{\mathbb Q} \to \bbQ$.

\begin{lemma}
$\Phi \subset X'$, $\Phi^{\vee}\subset Y \subset Y'$.
\end{lemma}
\begin{proof}
Since $nY \subset Y^{\sharp}$, $Y \subset Y'$; For any $\alpha \in \Phi$ and $y'\in Y'$, write $y'=\frac{1}{n}y^{\sharp}$ for some $y^{\sharp}\in Y^{\sharp}$, we have $<\alpha,y'>=\frac{1}{n}<\alpha,y^{\sharp}>=\frac{1}{n}(\alpha^{\vee},y^{\sharp})$ which is in $\bbZ$ because of $\alpha^{\vee} \in Y$ and the definition of $Y^{\sharp}$. 
\end{proof}

It follows that $$(X',\Phi, Y', \Phi^{\vee})$$ is a root datum associated to some (linear) connected reductive group $G'$ with a maximal torus $T'$ such that $X'$ is the character lattice of $T'$, $Y'$ is the cocharacter lattice of $T'$, $\Phi=\Phi(G',T')$ is the root system of $G'$, $\Phi^{\vee}=\Phi^{\vee}(G',T')$ is the coroots of $G'$.
There is a group homomorphism $\underline T \to T'$ induced by the inclusion map $Y \subset Y'$. 

As in setting of Section 3.2, we form the principal series representation of $G(\mathbb A_k)$ induced from $\nu\otimes\sigma$ where $\nu$ is an unramified automorphic character of $\underline T(\mathbb A_k)$ and $\sigma$ is an irreducible genuine automorphic representation of $T$ with a Weyl group invariant genuine automorphic central character. We also consider the principal series representation of $G'(\mathbb A_k)$, denoted $I'(\nu)$, induced from $\nu\otimes\sigma'$ where $\sigma'$ denotes the trivial representation of $T'(\mathbb A_k)$. We compare the product formulas (\ref{3.10}) for these groups. Let us fix a long root $\alpha$, and suppose $\nu(\underline h_{\alpha}(t))=|t|_{\mathbb A_k}^s$ for some complex number $s$ depending on $\alpha$ and $\nu$. Let $\mu_{\alpha,v}(s)$ be the Plancherel measure of $I(\nu_v)$ corresponding to $\alpha$ and $\mu'_{\alpha,v}(s)$ be the Plancherel measure of $I'(\nu_v)$ corresponding to $\alpha$. We have for the covering group $G$:
$$\left(\prod_{v\in S}\mu_{\alpha,v}(s)\right)\left(\prod_{v\notin S} d_{k_v} \right)\frac{L_S(ns+1)}{L_S(ns)}\frac{L_S(-ns+1)}{L_S(-ns)}=1.$$  
On the other hand, we have for the associated linear group $G'$:
$$\left(\prod_{v\in S}\mu'_{\alpha,v}(s)\right)\left(\prod_{v\notin S} d_{k_v} \right)\frac{L_S(s+1)}{L_S(s)}\frac{L_S(-s+1)}{L_S(-s)}=1.$$

Therefore, after making a simple substitution $s\mapsto ns$ in the second equation, we obtain
\begin{thm}\label{4.2} Let $\nu$ be an unramified character of $\underline T(k) \backslash\underline{T}(\mathbb A_k)$. Let $\sigma$ be a Weyl group invariant genuine irreducible representation of $T(\mathbb A_k)$ such that for all $v\notin S$, $\sigma_v(h_{\alpha}(t^n))=1$, $\forall \alpha\in\Phi, \forall t\in k_v^{\times}$. Let $\alpha$ be a long root of $\Phi$ such that $\nu(h_{\alpha}(t))=|t|_{\mathbb A_k}^s$, then 
$$\prod_{v\in S}\mu_{\alpha,v}(s)=\prod_{v\in S}\mu'_{\alpha,v}(ns)$$
for all but finitely many $s \in \mathbb C$.
\end{thm}

\subsection{Application to ramified representations of reductive groups over local fields}

\ref{4.2} gives us a way to go from those easy-to-compute Plancherel measures of linear groups to those unknown Plancherel measures of covering groups. We will make it precise in this section.

Let $p$ be any prime, and let $n$ be any positive integer. Let $F$ be a finite extension of $\mathbb Q_p$ containing $\mu_n$. Let $q=p^f$ be the order of its residue field $k_F$. Let $G$ be a $n$-fold cover of a split connected reductive group $\underline G$ defined over $F$. 
Suppose that we have an irreducible, genuine, Weyl group invariant representation $\sigma_p$ of the metaplectic torus $T$ and an unramified character $\nu_p$ of $\underline T$. Form the principal series representation of $G(F)$:
$$I(\nu_p)=Ind_B^G(\nu_p \otimes \sigma_p)$$ 
For any $w\in W$, let $$A_w(\nu_p):I(\nu_p) \to I(\nu_p^w)$$ be the standard intertwining operator. Then for generic $\nu_p$, $I(\nu_p),I(\nu_p^w)$ are irreducible and hence for any simple reflection $w=w_{\alpha}$, we have
$$A_{w_{\alpha}}(\nu_p^{w_{\alpha}})\circ A_{w_{\alpha}}(\nu_p)=1/\mu_{\alpha,p}(\nu_p)\cdot \mathrm{id}$$
Suppose that $\nu_p(\underline h_{\alpha}(t))=|t|_p^{s}$ for some $s \in \mathbb C$. Let $\mu_{\alpha,p}(s)=\mu_{\alpha,p}(\nu_p)$.

Our goal in this section is to find an explicit expression of $\mu_{\alpha,p}(s)$ in the case when the principal series $I(\nu_p)$ is ramified, i.e., $I(\nu_p)^{K_p}=0$. Let $k$ be a global field that has completions isomorphic to $F$ at places above $p$. Such $k$ is easy to be constructed. For instance, take $k_0$ to be a number field that has exactly one place $\frakp$ above $p$ such that $k_{\frakp}\cong F$. Then we can take $k$ to be the composite of $k_0$ and $\bbQ(\mu_n)$. The number field $k$ may have multiple places above $p$. Let $\chi_p$ be the central character of $\sigma_p$.

\begin{lemma}\label{4.3}
Let $v$ be a complex place $k$ and $\alpha$ be a simple root of $\underline G$. Suppose that $\nu_v(\underline h_{\alpha}(t))=|t|_{v}^s$ where $t\in\bbC$, $|\cdot|_v$ is the square of the complex absolute value and $s \in \bbC$. Then $\mu_{\alpha,v}(s)=-s^2/4\pi^2$.
\end{lemma}

\begin{thm}[Main Theorem]\label{main}
Let $\alpha\in\Delta$ be a long root. Suppose that there is a Weyl group invariant genuine ACC $\chi$ such that $\chi_v=\chi_p$ for all $v|p$, and for all but finitely many places $v$, $\chi_v(h_{\alpha}(t^n))=1$, $\forall \alpha\in\Phi, \forall t\in k_v^{\times}$. 
Suppose also that there is an unramified character $\nu$ of $\underline T(k)\backslash\underline T(\mathbb A_k)$ such that $\nu(\underline h_{\alpha}(t))=|t|_{\mathbb A}^s$. 
Then $$\mu_{\alpha,p}(s)=c\mu_{\alpha,p}'(ns)$$ for some constant $c$. Therefore, $\mu_{\alpha,p}(s)$ is a constant multiple of $$\frac{1-q^{-ns}}{1-q^{-ns-1}}\frac{1-q^{ns}}{1-q^{ns-1}}$$
\end{thm}

\begin{proof}
This will follow from \ref{4.2}, together with an elementary argument we are going to present. Since we are fixing $\alpha$, we write $\mu_{v}$ in place of $\mu_{\alpha,v}$ for any $v\in\Sigma$, the set of places of $k$. We may assume that $k$ has no real embeddings. Let $m$ be the number of places of $k$ above $p$. For any complex place $v$, by \ref{4.3}, $\mu_{v}(ns)/\mu_{v}(s)$ is a constant. So \ref{4.2} and the assumption gives
$$\mu_p(s)^m\prod_{v\in S_f-p}\mu_v(s)=C\cdot\mu'_p(ns)^m\prod_{v\in S_f-p}\mu'_v(ns)\:\:\:\:\:\:(1)$$ for some constant $C$. Here $S_f-p$ denotes the set of finite places in $S$ that is not above $p$.

By \ref{3.9}, we may write (up to constant) that 
$$\mu_p(s)=\prod_i \frac{(a_i-q^{-s})(a_i-q^{s})}{(b_i-q^{-s})(b_i-q^{s})}$$ for finitely many complex numbers $a_i$, $b_i$, counting multiplicities. Choose for each $i$ a complex number $s_i$ so that $a_i=q^{s_i}$. Then the zeros of $\mu_p(s)$ are
$$\pm s_i+\frac{2\pi i}{\log q}\mathbb Z=\pm s_i+\frac{2\pi i}{f\log p}\mathbb Z$$
On the other hand, by \ref{2.6}, the zeros of the right hand side of (1) are
$$\frac{2\pi i}{n\log q_v}\mathbb Z=\frac{2\pi i}{nf_v\log p_v}\mathbb Z$$ for some $v \in S_f$.

We claim that 
$$s_i\in \frac{2\pi i}{nf\log p}\mathbb Z$$ for all $i$. 
In fact, for a fixed $i$, by Pigeonhole Principal, there is some prime $p'$ such that there are infinitely many elements in $s_i+\frac{2\pi i}{f\log p}\mathbb Z$ that belong to $\frac{2\pi i}{nf'\log p'}\mathbb Z$. 
If $p'\neq p$, take two distinct elements in 
$$\left(s_i+\frac{2\pi i}{f\log p}\mathbb Z\right)\cap\left(\frac{2\pi i}{nf'\log p'}\mathbb Z\right)$$ Because their difference also belongs to $\frac{2\pi i}{nf'\log p'}\mathbb Z$, we obtain a nontrivial $\bbQ$-linear relation among $\log p$ and $\log p'$, which is impossible since $\log p, \log p'$ are linearly independent over $\bbQ$. Therefore, $p'=p$ and in particular, $s_i\in \frac{2\pi i}{nf\log p}\mathbb Z$, proves the claim. 

It follows that $a_i^n=1$ and all the zeros of $\mu_p(s)$ belong to $$\frac{2\pi i}{nf\log p}\mathbb Z$$
By the same argument, $b_i^n=q$ and all the poles of $\mu_p(s)$ belong to $$\pm \frac{1}{n}+\frac{2\pi i}{nf\log p}\mathbb Z$$ Therefore the zeros and poles of $\mu_p(s)^m$ are contained in that of $\mu'_p(ns)^m$. It follows that $\mu_p(s)^m$ is a constant multiple of $\mu'_p(ns)^m$, i.e., $(\mu_p(s)/\mu'_p(ns))^m$ is a constant. On the other hand, by \ref{3.9}, both of $\mu_p(s)$ and $\mu'_p(ns)$ are positive on the imaginary axis, so $\mu_p(s)/\mu'_p(ns)$ must equal to a positive constant when restricted to the imaginary axis. It follows that $\mu_p(s)/\mu'_p(ns)$ is a constant.  
\end{proof}

\begin{remark}
In certain cases, the assumption in \ref{main} can be achieved using the construction of a Weyl group invariant genuine ACC in Section 5.
\end{remark}

\section{Construction of Weyl group invariant genuine automorphic central characters}

\subsection{$n$-fold cover of $SL_2$}

We let $m$ to be $n$ when $n$ is odd and $n/2$ when $n$ is even. Then $Y^{\sharp}=mY \cong m\mathbb Z$ and hence both $\underline T(F)$ and $\underline T^{\sharp}(F)$ are isomorphic to $F^{\times}$. The natural inclusion $Y^{\sharp} \to Y$ induces a map $\underline T^{\sharp}(F) \to \underline T(F)$, $t \to t^m$. The covering torus $T(F)$ is generated by symbols $h(t)$, $t\in F^{\times}$ satisfying $h(t)h(u)=(t,u)_nh(tu)$. Let $\pi: T(F)\to \underline T(F)$, $h(t) \mapsto t$ be the covering map. We define $T^{\sharp}(F)$ to be the set $F^{\times} \times \mu_n$ with group law
$$(t,\zeta)(t',\zeta')=(tt',\zeta\zeta')$$ when $n\equiv 0,\pm 1 (4)$;
$$(t,\zeta)(t',\zeta')=(tt',\zeta\zeta'(t,t')_2)$$ when $n\equiv 2(4)$. Under this group law, $T^{\sharp}(F)$ is an abelian group. Moreover, $T^{\sharp}(F) \to \underline T^{\sharp}(F)$, $(t,\zeta) \mapsto t$ is an $n$-to-$1$ homomorphism. We then define a map 
$T^{\sharp}(F) \to T(F)$, $(t,\zeta) \mapsto \zeta h(t^m)$, then it is a group homomorphism with image $Z(F)$, the center of $T(F)$. It follows that $Z(F) \cong T^{\sharp}(F)/(\mu_m \times 1)$.  

Let $\gamma: F^{\times}/(F^{\times})^2 \to \mu_{\infty}$ be a function such that
$$\gamma(a)\gamma(b)=(a,b)_F\gamma(ab)$$
and $$\gamma(1)=1,$$
where $(a,b)_F$ is the quadratic Hilbert symbol.

We now define a character of $T^{\sharp}(F)$ as follows, when $n\equiv 0,\pm 1 (4)$, $(t,\zeta) \mapsto \zeta$; when $n\equiv 2 (4)$, $(t,\zeta) \mapsto \zeta\gamma(t)$. Then it factors through $Z(F)$ because the Weil index $\gamma$ is trivial on $\mu_m=(\mu_n)^2$ when $n$ is even. It is easy to see that this character is Weyl group invariant. We denote this character by $\gamma^*_p$ where $p$ indicates the residual characteristics of $F$.

\begin{prop}
Any Weyl group invariant genuine character of $T^{\sharp}(F)$ is a product of $\gamma^*_p$ and a quadratic character.
\end{prop}
\begin{proof}
Take two genuine Weyl group invariant characters $\chi,\chi'$, let $\varepsilon=\chi/\chi'$. Then $\varepsilon$ factors through $\underline T^{\sharp}=F^{\times}$. Since $\varepsilon$ is Weyl group invariant, $\varepsilon(t)=\varepsilon(t^{-1})$ and so $\varepsilon(t)^2=1$.
\end{proof}

Let $\nu_p: \underline T(F) \to \bbR_{>0}$ be the character defined by $\nu_p(t)=|t|_F^s$ for $t\in F^{\times}$. Let $\sigma_p$ be the genuine irreducible representation of $T(F)$ corresponding to $\gamma_p^*$. Put 
$$I(s)=I(\nu_p)=\Ind_B^G(\nu_p \otimes \sigma_p)$$

\begin{lemma}\label{5.2}
The assumptions in \label{4.4} are satisfied in the case of $n$-fold cover of $SL_2$.
\end{lemma}
\begin{proof}
We take a system $\{\psi_v\}_v$ of additive characters of $k_v$, and let $\gamma_{v}=\gamma_{\psi_v}$ be the Weil index attached to the quadratic form $x^2-y^2$ and $\psi_v$. Let $\gamma_v^*$ be the genuine central character of $T^{\sharp}(k_v)$ associated to $\gamma_v$ as described above. The tensor product
$$\gamma_{\mathbb A}^*:=\prod_v \gamma_v^*$$ defines a genuine character of the center $\underline T(k) \backslash T^{\sharp}(\mathbb A_k)$. By construction, $\gamma_v^*=\gamma_p^*$ for any $v|p$. When $n$ is odd, $\gamma_v^*(h(t^n))=1$ for $t\in k_v^*$ by definition of $\gamma_v^*$; when $n$ is even, this is also true since $t^n$ is a square in $k_v^*$ and Weil index is trivial on squares. 
\end{proof}

\begin{cor}\label{5.3}
$I(0)$ is irreducible, $I(1/n)$ and $I(-1/n)$ are reducible.
\end{cor}
\begin{proof}
Theorem \label{4.4} implies that $$\mu_p(s)=c\frac{1-q^{-ns}}{1-q^{-ns-1}}\frac{1-q^{ns}}{1-q^{ns-1}}$$ for some constant $c$. In particular, it has a double zero at $s=0$ and has a simple pole at $s=\pm 1/n$. It follows that $I(1/n)$ and $I(-1/n)$ are reducible. By \label{6.3} in the appendix, $I(0)$ is irreducible since $\mu_p^{-1}(s)$ has a pole at $s=0$.
\end{proof}

\subsection{Two-fold cover of simply-connected Chevalley group of type ADE}
In this section, we assume that $\underline G$ is a simply-connected, simply-laced Chevalley group defined over a local field $F$. Let $G$ be its nontrivial two-fold central extension. Let $T$ be the metaplectic torus of $G$ generated by $h_{\alpha}(t)$ for $\alpha\in\Phi$, $t\in F^{\times}$. Let $Z_T$ be the center of $T$ and $T^2$ be the subgroup of $Z_T$ generated by $h_{\alpha}(t^2)$ for $t\in F^{\times}$, $\alpha\in \Phi$.

The following proposition follows directly from Proposition 4.2 of \cite{Sav04}.
\begin{prop}
Let $S$ be a subset of vertices of the Dynkin diagram of $G$ such that no two vertices in $S$ are adjacent and every vertex not in $S$ is adjacent to an even number of vertices in $S$. 
Then there is no nonempty $S$ if $\Phi$ is of type $A_{2n},E_6,E_8$, there is exactly one nonempty $S$ when $\Phi$ is of type $A_{2n-1},D_{2n-1},E_7$ and there are three nonempty $S$ when $\Phi$ is of type $D_{2n}$. Moreover, $$Z(T)/T^2$$ is generated by 
$$h_S(t)=\prod_{\alpha\in S}h_{\alpha}(t)$$ $t\in F^{\times}/(F^{\times})^2$.
\end{prop}

Let us label the Dynkin diagram of $D_{2n}$ as follows: we label the nodes from the very left end to the conjunction by $1,2,\cdots,2n-2$ respectively, and we label the node at the top right of the conjunction by $2n-1$, the node at the bottom right of the conjunction by $2n$. The the three nonempty $S$ are 
$$S_1=\{ 1,3,\cdots,2n-3,2n-1 \}$$
$$S_2=\{ 1,3,\cdots,2n-3,2n \}$$
$$S_3=\{ 2n-1,2n \}$$
Thus every element in $Z_T/T^2$ is of the form
$$\pm h_{S_1}(t_1)h_{S_2}(t_2)h_{S_3}(t_3)$$

\begin{cor}\label{5.5}
When $\Phi$ is of type $A_{2n},E_6,E_8$, $Z_T/T^2$ is trivial; When $\Phi$ is of type $A_{2n-1},D_{2n-1},E_7$, $Z_T/T^2$ is generated by $\mu_2$ and $h_S(t)$, $t \in F^{\times}/(F^{\times})^2$ satisfying $h_S(t)h_S(u)=(t,u)^{|S|}h(tu)$; When $\Phi$ is of type $D_{2n}$, $Z_T/T^2$ is generated by $h_{S_1}(t)$, $h_{S_2}(t)$, $h_{S_3}(t)$.
\end{cor}

\begin{prop}\label{5.6}
Let $\gamma:F^{\times}/(F^{\times})^2 \to \mu_4$ be a function such that $$\gamma(t)\gamma(u)=(t,u)\gamma(tu)$$ and $$\gamma(1)=1.$$ (For instance, the normalized Weil index attached to the quadratic form $x^2-y^2$ and an additive character $\psi: F \to S^1$) 
Define $$\gamma^*(\epsilon h_S(t))=\epsilon\gamma(t)^{|S|}$$
Then $\gamma^*$ is a genuine character of $Z_T$.
\end{prop}

\begin{prop}
The Weyl group $W$ fixes $Z_T/T^2$ pointwisely.  
\end{prop}
\begin{proof}
It is enough to show that $W$ fixes $h_S(t)$ for any $S$ and $t\in F^{\times}/(F^{\times})^2$, that is, $$w_{\alpha}(1)h_S(t)w_{\alpha}(-1)=h_S(t)$$ $\forall \alpha\in\Delta, t\in F^{\times}/(F^{\times})^2$. By the definition of $S$, if $\alpha$ does not belong to $S$, then it connects to an even number of nodes in $S$, if the number is zero, we are done; if not, then it must connect to two nodes in $S$, say $\beta,\gamma$. It suffices to check that
$$w_{\alpha}(1)h_{\beta}(t)h_{\gamma}(t)w_{\alpha}(-1)=h_{\beta}(t)h_{\gamma}(t)$$ The LHS is 
$$(t,c(\alpha,\beta))h_{\alpha+\beta}(t)(t,c(\alpha,\gamma))h_{\beta+\gamma}(t)=(t,c(\alpha,\beta)c(\alpha,\gamma))h_{\alpha+\beta}(t)h_{\alpha+\gamma}(t)$$
$$=(t,c(\alpha,\beta)c(\alpha,\gamma))(t,c(\beta,\alpha)t)h_{\beta}(t)h_{\alpha}(t)(t,c(\alpha,\gamma)t)h_{\alpha}(t)h_{\gamma}(t)$$
$$=(t,c(\alpha,\beta)c(\beta,\alpha))(t,t)h_{\beta}(t)h_{\gamma}(t)=(t,-1)(t,t)h_{\beta}(t)h_{\gamma}(t)=h_{\beta}(t)h_{\gamma}(t)$$

Finally, if $\alpha\in S$, then since $$w_{\alpha}(1)h_{\alpha}(t)w_{\alpha}(-1)=h_{\alpha}(t^{-1})=h_{\alpha}(t)$$ and $\alpha$ is not adjacent to any element in $S$, $w_{\alpha}(1)$ fixes $h_S(t)$.
\end{proof}

For a number field $k$, we take a system $\{\psi_v\}_v$ of additive characters of $k_v$ for all places $v$ of $k$, and we let $\gamma_{v}=\gamma_{\psi_v}$ be the Weil index attached to the quadratic form $x^2-y^2$ and $\psi_v$.

\begin{lemma}[\cite{Wei64}]
$\gamma_v$ is unramified for almost all $v$ and $$\prod_v \gamma_v(x)=1$$ $\forall x \in k^{\times}$.
\end{lemma}

Let $\gamma_v^*$ be the genuine central character associated to $\gamma_v$, we form the product
$$\gamma_{\mathbb A}^*:=\prod_v \gamma_v^*$$ which defines a character of the center $Z_T(\mathbb A_k)$ of $T(\mathbb A_k)$. 

\begin{lemma}
$$\underline T(k) \cap Z_T(\mathbb A_k)$$ is generated by $$\prod_{\alpha\in \Delta} h_{\alpha}(t_{\alpha}^2)$$ and 
$$h_S(t)=\prod_{\alpha\in S} h_{\alpha}(t)$$ where $t_{\alpha},t \in k^{\times}$.
\end{lemma}
\begin{proof}
Any element in the intersection is generated by the elements 
$$\prod_{\alpha\in \Delta} h_{\alpha}(u_{\alpha})$$
$$\prod_{\alpha\in S} h_{\alpha}(t_{\alpha})$$
such that $u_{\alpha},t_{\alpha} \in k^{\times}$, $u_{\alpha}$ is a square in $k_v$ for all $v$ and $\alpha\in\Delta$, and $t_{\alpha},t_{\beta}$ differ by a square in $k_v$ for all $\alpha,\beta\in S$ and places $v$. But local-global principal implies that any element of $k$ that is a square in each $k_v$ must be a square in $k$, proves the lemma.  
\end{proof}

\begin{cor}\label{5.10}
$\gamma_{\mathbb A}^*: Z_T(\mathbb A_k) \to \mu_4$ is trivial on $\underline T(k) \cap Z_T(\mathbb A_k)$, hence is a genuine Weyl group invariant ACC.
\end{cor}

\begin{remark}
By the proof of \ref{5.2}, the assumptions of \ref{4.4} are also met if we take the principal series representation to be one induced from the irreducible genuine representation of $T(\mathbb A_k)$ associated to $\gamma_{\mathbb A}^*$.
\end{remark}

\subsection{Two-fold cover of $GL_n$}
Consider the homomorphism $i: GL_n(F) \to SL_{n+1}(F)$ given by mapping $A \in GL_n(F)$ to $\begin{pmatrix}
A & 0 \\ 0 & det(A)^{-1}.  
\end{pmatrix}$
Let $\widetilde{SL}_{n+1}(F)$ denote the nontrivial two-fold central extension of $SL_{n+1}(F)$, and let $\widetilde{GL}_{n}(F)$ denote its pullback, which is a nontrivial two-fold central extension of $GL_{n}(F)$. Because the standard maximal torus of $GL_{n}(F)$ is isomorphic to the standard maximal torus of $SL_{n+1}(F)$ under the map $i$, their pullbacks are also isomorphic. So we may identify the covering torus of $\widetilde{GL}_{n}(F)$ with the covering torus of $\widetilde{SL}_{n+1}(F)$, which is generated by $h_i(t)=h_{\alpha_i}(t)$ for $1\leq i \leq n$ and $t \in F^{\times}$. By \ref{5.5}, $Z_T(F)$ is generated by $\mu_2$ and $h_i(t^2)$ for $1\leq i \leq n$, $t \in F^{\times}$ when $n$ is odd, and is generated by $h_i(t^2)$ and $h_{odd}(t)=h_1(t)h_3(t)\cdots h_n(t)$ for $1\leq i \leq n$, $t \in F^{\times}$ when $n$ is even. Globally, we may define $\widetilde{GL}_{n}$ over the adeles and its covering torus in the same fashion. 

\begin{prop}
There exists a Weyl group invariant genuine character $$\gamma^*:Z_T(\mathbb A_k)/T^2(\mathbb A_k) \to \mu_4$$ of the center of the covering torus of $\widetilde{GL}_{n}(\mathbb A_k)$. Explicitly, $\gamma^*$ is given by the isomorphism $Z_T(\mathbb A_k)/T^2(\mathbb A_k) \cong \mu_2$ when $n$ is odd, and is given by $$\gamma^*(\epsilon h_{odd}(t))=\epsilon \gamma(t)^{(n+1)/2}$$ when $n$ is even.  
\end{prop}
\begin{proof}
By \ref{5.10}, there is a Weyl group invariant genuine ACC of the covering torus of $\widetilde{SL}_{n+1}(\mathbb A_k)$ where Weyl group refers to the Weyl group of $SL_{n+1}$. Pulling back through $i$, we get a genuine ACC of the covering torus of $\widetilde{GL}_{n}(\mathbb A_k)$. Because the Weyl group of $GL_n$ may be identified with the subgroup of the Weyl group of $SL_{n+1}$ that leaves the last entry of the torus invariant, the character is invariant under the action of the Weyl group of $GL_n$. The second part follows directly from \ref{5.6}.   
\end{proof}

\section{Appendix}

The results in this appendix are due to Gordan Savin.

\subsection{Introduction}
Many of Harish-Chandra's results on intertwining operators and reducibility 
of induced representations are contained in his work on the Plancherel measure 
for reductive groups. 
However, since Harish-Chandra works with algebraic groups, his results are not 
applicable to non-linear central extensions of reductive groups. A purpose of this note 
is to provide simple proofs of some of those results in a manner applicable to 
non-linear groups. 

\subsection{Notation} 
Let $k$ be a $p$-adic field with residual field of order $q$.  Fix a uniformizing element 
$\varpi$ and an absolute value on $k$ such that $|\varphi|=1/q$. 

Let $G$ be a reductive group defined over $k$. 
Let $A$ be a maximal split torus and $\Delta$ a set of simple roots for 
the corresponding root system. For every $\Theta\subseteq \Delta$ we have a standard 
parabolic subgroup $P_{\Theta}=M_{\Theta}N_{\Theta}$ (or simply $P=MN$ if 
$\Theta$ is fixed). In particular, $P_{\emptyset}$ is the minimal parabolic subgroup, while 
maximal parabolic subgroups correspond to $\Theta$ such that $\Delta\setminus \Theta$ 
has one simple root. 
Let $A_{\Theta}\subseteq A$ be the connected 
component of the intersection of kernels of all $\alpha\in\Theta$. Then $A_{\Theta}$ is a 
split torus contained in the center of $M_{\Theta}$ and 
$A_{\Theta}\cong X_{\Theta}\otimes k^{\times}$ where $X_{\Theta}$ is the co-character lattice. 
Let $W$ be the Weyl group  and $W_{\Theta}\subseteq W$ the subgroup generated by 
simple reflections in $\Theta$. Let $w^{\ell}$ and $w^{\ell}_{\Theta}$ be the longest elements in 
$W$ and $W_{\Theta}$ respectively. Let ${w_{\Theta}}=w^{\ell}w^{\ell}_{\Theta}$. Then 
$\bar{\Theta}={w_{\Theta}}(\Theta)\subseteq \Delta$ and $P_{\bar{\Theta}}$ (or simply $\bar P$) is the
associated parabolic subgroup. 

Now assume that $P_{\Theta}$ is maximal and $P_{\Theta}={P_{\bar\Theta}}$. 
Then $w_{\Theta}$ is of order $2$ and  normalizes $M_{\Theta}$. Let $Y_{\Theta}\cong \bbZ$ be the 
lattice of characters of $M_{\Theta}$ trivial on the center of $G$. Fix a generator 
$\chi$ of this group of characters. Then $\chi^{w_{\Theta}}=\chi^{-1}$.  The image of natural pairing 
between $X_{\Theta}$ and $Y_{\Theta}$ is $m\bbZ$, for some positive integer $m$. 
Let $a_{\Theta}\in A_{\Theta}$ be such that $\chi(a_{\Theta})=\varpi^m$. 

\subsection{Induced representations}

Let $P=MN$ be a maximal parabolic subgroup corresponding to $\Theta\subset\Delta$. 
Let $\delta_N$ be the modular character. 
Let $(\pi,E)$ be a supercuspidal representation of $M$. Let $I_P^G(\pi)$ be the (normalized)
induced representation. It consists of locally constant functions $f: G\rightarrow E$ 
such that $f(nmg)=\delta_N^{1/2}(m)\pi(m)f(g)$ for all $n\in N$, $m\in M$ and $g\in G$.   
   
If $\bar{P}\neq P$ then $I_P^G(\pi)_N\cong \delta_N^{1/2}\pi$ (the Jacquet functor) and 
$I_P^G(\pi)$ is always irreducible. Henceforth we assume that $\bar{P}=P$. For example, 
this is true for every maximal parabolic in $Sp_{2n}(k)$. Using the Bruhat decomposition 
$G=\cup_{w\in W_{\Theta}\backslash W/W_{\Theta}} PwP$ define a filtration of 
$I_P^G(\pi)$ by $P$-invariant subspaces  
\[ 
J(\pi) \subseteq J'(\pi) \subseteq I_P^G(\pi)
\] 
 where $J(\pi)$ is subspace consisting of functions in $I_P^G(\pi)$  
 supported on $Pw_{\Theta}P=Pw_{\Theta}N$, and $J'(\pi)$ is the subspace of functions 
supported on $G\setminus P$. Then $J(\pi)_N\cong \delta_N^{1/2}\pi^w$ by integrating over $w_{\Theta}N$,  
$(J'(\pi)/J(\pi))_N=0$, and $(I_P^G(\pi)/J'(\pi))_N\cong \delta_N^{1/2}\pi$ by evaluating functions at the identity of $G$.  
 
 The representation $I_P^G(\pi)$ sits in a family of induced representations. More 
 precisely, for every  $s\in\mathbb C$ let $(\pi_{s},E)$ be the twist of $\pi$ by $|\chi|^s$.
   This representation depends only on $q^{s}\in \bbC^{\times}$. In this way we have 
 an algebraic family $\mathcal F$ of cuspidal representations
 of $M$. From now on, rather than fixing one particular member of $\mathcal F$,
 $\pi$ will usually denote any member of $\mathcal F$. However, if $\pi$ is fixed, 
 then we can identify $\mathcal F$ with $\bbC^{\times}$ or with $\bbC$, if it is more convenient. 

 Let $K$ be a special maximal compact subgroup of $G$. 
Since $G=PK$, by restricting functions in $I_P^G(\pi)$ to $K$, we can identify
all induced representations in $\mathcal F$ with $X=\Ind^{K}_{K\cap P}(E)$. 
In particular, an element $f\in X$  defines a map $\pi\mapsto
f(\pi)\in I_P^G(\pi)$ for all $\pi\in \mathcal F$. We shall call this map a 
constant section. A regular section is a linear combination $\sum a_{i}f_{i}$ where 
$a_{i}$ are regular functions on $\mathcal F$, and $f_{i}$ are constant sections. 
If $f$ is a constant section such that 
 $f(\pi)$ is in $J(\pi)$ (or in $J'(\pi)$) for 
one $\pi\in \mathcal F$, then $f(\pi)$ is in $J(\pi)$ (or in $J'(\pi)$, respectively) for
all $\pi$. Thus we have a corresponding filtration of constant sections
\[ 
Y\subseteq Y'\subseteq X.
\]

 We have a rational family of intertwining operators 
 $A(\pi):I_P^G(\pi) \rightarrow I_P^G(\pi^{w_{\Theta}})$ 
 traditionally defined (on an open set in $\mathcal F$) by the integral
 \[ 
 A(\pi)f(g)=\int_N f(w_{\Theta}ng) ~dn. 
 \] 
 More precisely, there exists a rational function $a$ on $\mathcal F$ such that for every 
 constant section $f\in X$, $\pi\mapsto a(\pi)\cdot A(\pi)(f)$ is a regular section. 
Via the Frobenius reciprocity, the intertwining operator $A(\pi)$ corresponds to
a $P$-intertwining map 
\[ 
\ell(\pi): I_P^G(\pi) \rightarrow \delta_N^{1/2}\pi^{w_{\Theta}}
\] 
 which, when restricted to $J(\pi)$, is given by the integral over ${w_{\Theta}}N$ and 
 is therefore non-zero. We view $\ell$ as a rational map $\pi\mapsto \ell(\pi)\in \Hom_{\bbC}(X,E)$. 
  The analytic behaviour of $A$ and $\ell$ is the same, in particular: 
\begin{enumerate} 
\item For every 
constant section $f\in X$ the map $\pi\mapsto a(\pi)\cdot \ell(\pi)(f)\in E$ 
is regular. 
\item  For every constant section $f\in Y$
the map $\pi\mapsto  \ell(\pi)(f)\in E$ is regular. 
\item  For every $\pi$ 
there is $f\in Y$ such that $\ell(\pi)(f)\neq 0$. 
\end{enumerate}  

\begin{lemma}\label{L1} We have: 
\begin{enumerate} 
\item For every constant section $f\in Y'$ the map 
$\pi\mapsto \ell(\pi)(f)\in E$ is regular. 
\item  
If $\pi$ is not isomorphic to $\pi^{w_{\Theta}}$ (which is always true if $\pi$ is not
 unitary) then  for every constant section $f\in X$, the map 
$\pi\mapsto  \ell(\pi)(f)\in E$ is regular at $\pi$. 
\end{enumerate} 
\end{lemma} 
\begin{proof} 
1)  Assume that $\ell$, when applied to all sections in $Y'$, has a pole of 
(maximal) order $n$ at $\pi$. Using $\pi$ as a basepoint, 
we identify the family $\mathcal F$ with $\bbC$. Then $s^n \ell$ defines
 a non-zero $P$-intertwining map from $J'(\pi)/J(\pi)$ to $\delta_N^{1/2}\pi^{w_{\Theta}}$. But there is no
 such map since $(J'(\pi)/J(\pi))_N=0$. This is a contradiction. 2) Similarly
if $\ell$, when applied to all sections in $X$, has a pole of order $n$ at $\pi$, then
$s^n \ell$ defines a non-zero $P$-intertwining map from $I_P^G(\pi)/J'(\pi)$ to
$\delta_N^{1/2}\pi^{w_{\Theta}}$. Since $(I_P^G(\pi)/J'(\pi))_N\cong \delta_N^{1/2}\pi$, we obtain an
isomorphism of $\pi$ and $\pi^{w_{\Theta}}$. This is a contradiction. The lemma is proved.
\end{proof} 

 Since $I_P^G(\pi)$ is irreducible for a generic $\pi\in \mathcal F$, 
 the composition $A(\pi^{w_{\Theta}})\circ A(\pi)$ is a necessarily equal to 
 $\mu^{-1}(\pi) \cdot 1_{X}$ for a 
 rational function $\mu(\pi)$, called Plancherel measure. 
The following is contained in Casselman's notes: 
\begin{prop} We have: 
\begin{enumerate}
 \item If $\pi$ is not unitary, then $I_P^G(\pi)$ is reducible iff $\mu^{-1}(\pi)=0$.
 \item If $\pi$ is unitary and $\pi$ is not isomorphic to $\pi^{w_{\Theta}}$ then
 $I_P^G(\pi)$ is irreducible.
\end{enumerate}
\end{prop}
\begin{proof}
1) Assume that  $I_P^G(\pi)$ reduces. We claim that $I_P^G(\pi)$ has a non-split
composition series of length 2.
 Let $V'$ be an irreducible submodule. Then,
by the Frobenius reciprocity, $\delta_N^{1/2}\pi$ must be a quotient of $V'$. Since
$\pi$ is not isomorphic to $\pi^w$, $\delta_N^{1/2}\pi$  appears with
multiplicity one in $I_P^G(\pi)_N$.  This shows that $V'$ is unique submodule,
$V'_N\cong \delta_N^{1/2}\pi$ and  $V''_N\cong \delta_N^{1/2}\pi^w$ where
 $V''=I_P^G(\pi)/V'$.
It follows, from the Frobenius reciprocity, that $V''$ is a submodule of $I_P^G(\pi^w)$ and $V'$ is a
quotient. Thus $A(\pi^w)\circ A(\pi)=0$. Conversely,
since $A(\pi)$ and $A(\pi^w)$ are defined and non-zero, $\mu^{-1}(\pi)=0$ must imply that
$I_P^G(\pi)$ and $I_P^G(\pi^w)$ both reduce.
2) If $I_P^G(\pi)$ is reducible then, arguing as in 1), it has a unique submodule and a
unique quotient. But $I_P^G(\pi)$ is unitary and thus semisimple. This is a contradiction.
\end{proof} 

Assume that $\pi$ is unitary and that $\pi$ is isomorphic to $\pi^{w_{\Theta}}$.  Let $t_w : \pi^{w_{\Theta}}\rightarrow \pi$ be a map intertwining the two actions of $M$. 
Let $\pi_s=\pi\otimes |\chi |^s$, where $s\in \bbC$. Let $I(s)$ be the induced representation of $G$.  We have an intertwining map 
 $A(s):I_P^G(s) \rightarrow I_P^G(-s)$ defined by the integral
 \[ 
 A(s)f(g)=t_w \int_N f(w_{\Theta}ng) ~dn. 
 \] 
Let $\mu^{-1}(s)$ be the function obtained by composing $A(s)$ and $A(-s)$. The following is not contained in Casselman's notes: 
 
\begin{prop}\label{6.3} Assume that $\pi$ is unitary and that $\pi$ is isomorphic to $\pi^{w_{\Theta}}$. Then 
the following three are equivalent: 
\begin{enumerate}
 \item $I_P^G(\pi)$ is irreducible.
 \item The exact sequence 
\[
0\rightarrow \delta_N^{1/2}\pi \rightarrow I_P^G(\pi)_N \rightarrow \delta_N^{1/2}\pi \rightarrow 0
\] 
 does not split.  
  \item The function $\mu^{-1}(s)$ has a pole,  
 necessarily of order 2, at $s=0$. 
\end{enumerate}
\end{prop} 
\begin{proof}
 Since $I_P^G(\pi)$ is unitary, it is irreducible if and only if
  $\Hom_G(I_P^G(\pi), I_P^G(\pi))\cong \bbC$. 
By the Frobenius reciprocity, $\Hom_G(I_P^G(\pi), I_P^G(\pi) )\cong 
\Hom_M(I_P^G(\pi),  \delta_N^{1/2}\pi)$. Thus $I_P^G(\pi)$ is irreducible if and only
 if $I_P^G(\pi)_N$ does not split. 
This proves the equivalence of 1) and 2). 
 If $\ell$ is regular at $\pi$ then it defines a splitting of the exact sequence, 
 and $I_P^G(\pi)$ reduces. In order to finish, we need to show that if $\mu^{-1}$ has a pole at 
$\pi$, then it is of order $2$ and $I_P^G(\pi)_N$ does not split.

 Let $a_{\Theta}$ be the element in the center of $M$, defined in the previous section. Let $\epsilon$ be the central character of $\pi$.  
 Let $s\mapsto f_{s}\in I_P^G(\pi_s)$ be a constant section. 
 Then $$f_{s}'=f_{s}-\epsilon^{-1}(a_{\Theta})
 \delta_{N}^{-1/2}(a_{\Theta})|\chi(a_{\Theta})|^{-s}\cdot \pi_{s}(a_{\Theta})(f_{s})$$ 
is a regular section such that $f'_{s}\in J'(\pi_s)$ for all $s$. In particular, by Lemma \ref{L1},  
  $s\mapsto\ell(f_{s}')$ is regular. If $s\neq 0$ but $|s|$ is sufficiently small, then $\ell$ is 
 defined on the whole $I_P^G(\pi_s)$, hence 
 \[ 
 \ell(f_{s}')=\ell(f_{s})-q^{2ms}\ell(f_{s})=(1-q^{2ms})\ell(f_{s}). 
 \] 
 Solving for $\ell(f_{s})$ we see that $\ell$ can have the pole of order at most one at $s=0$, 
 and that happens if and only if $\ell(\pi)(f_{0}')\neq 0$ for some constant  
 section $f_{s}$. Now we can finish 
 easily. Assume that $\ell$ has a pole at $s=0$. We want to show that $I_P^G(\pi)_N$ does not split. 
If it splits, then $\ell(\pi)$ extends to    
 $I_P^G(\pi)$. Let $\tilde \ell(\pi)$ be an extension. Then,  
 for every constant section $f_s$, the above computation shows $\ell(\pi)(f_{0}')= (1-q^0)\tilde\ell(\pi)(f_{0})=0$, 
a contradiction. 
Thus we have seen that the intertwining operator $A$ can have a pole of order at most one at 
$I_P^G(\pi)$, and if it does, then $I_P^G(\pi)$ is irreducible. It follows that the residue of $A$ 
acts on $I_P^G(\pi)$ as multiplication by a non-zero scalar, so $\mu^{-1}$ has a double pole.
\end{proof}

\begin{cor} Assume that $\pi$ is unitary, $\pi$ is isomorphic to $\pi^{w_{\Theta}}$, 
 and $I_{P}^{G}(\pi)$ reduces 
(necessarily into two non-isomorphic summands). Then 
$\mu(0)^{-1}\neq 0$ and $A(0)$ acts on the two summands by different scalars 
$\pm \mu(0)^{-1/2}$.
\end{cor}
\begin{proof} We know that $A(0)$ is well defined and non-zero since 
$\ell(\pi)$ is non-zero on $J(\pi)$. Let $\lambda_{1}$ and $\lambda_{2}$
be the eigenvalues of $A(0)$ on the two summands. Since  $\ell(\pi)$
is not given by evaluation of functions at identity, the operator $A(0)$
 is not a multiple of the identity operator. Thus $\lambda_{1}\neq \lambda_{2}$. Since 
 $A(0)^{2}=\mu(\pi)^{-1}$, the eigenvalues must be non-zero and of opposite signs. 
 \end{proof}

\noindent
{\bf Remark:} If $G'$ is a central extension of $G$ by a group of $n$-th roots of 1, then 
all of the above arguments go through provided that we replace $a_{\Theta}$ by an 
inverse image in $G'$ of  $a^{n}_{\Theta}$. (That element is clearly contained in the 
 in the center of $M'$.)

\end{document}